\documentclass{article}

\usepackage[final]{nips_2016}
\usepackage{color}
\usepackage[utf8]{inputenc} 
\usepackage[T1]{fontenc}    
\usepackage{hyperref}       
\usepackage{url}            
\usepackage{booktabs}       
\usepackage{amsmath,amsfonts,amssymb}       
\usepackage{nicefrac}       
\usepackage{microtype}      
\usepackage{amsthm}
\usepackage{mathtools}

\DeclarePairedDelimiter\abs{\lvert}{\rvert}

\newcommand{\matr}[1]{\mathbf{#1}}

\newtheorem{lemma}{Lemma}
\newtheorem{coro}{Corollary}[section]
\title{Saddle-free Hessian-free Optimization}

\author{
  Martin Arjovsky\\
  Courant Institute of Mathematical Sciences\\
  New York University\\
  \texttt{martinarjovsky@gmail.com} \\
}

\begin{document}

\maketitle

\begin{abstract}
  Nonconvex optimization problems such as the ones in 
  training deep neural networks suffer
  from a phenomenon called saddle point proliferation. This means that there
  are a vast number of high error saddle points present
  in the loss function. Second
  order methods have been tremendously successful and widely adopted in the
  convex optimization community, while their usefulness in deep learning
  remains limited. This is due to two problems: computational complexity and
  the methods being driven towards the high error saddle points. We introduce
  a novel algorithm specially designed to solve these two issues, providing a
  crucial first step to take the widely known advantages of Newton's
  method to the nonconvex optimization community, especially in high
  dimensional settings.
\end{abstract}

\section{Introduction}
  The loss functions arizing from learning deep neural networks are incredibly
  nonconvex, so the fact that they can be successfully optimized in a lot
  of problems remains a partial mistery.
  However, some recent work has started to shed light on this issue \citep{
  Auff-ea-SG, Chro-ea-LM, Dauphin-ea-SP}, leading to three likely conclusions:
\begin{itemize}
  \item There appears to be an exponential number of local minima.
  \item However, all local minima lie within a small range of error
    with overwhelming proability. Almost all local minima
    will therefore have similar error to the global minimum.
  \item There are exponentially more saddle points than minima, a phenomenon
    called saddle point proliferation.
\end{itemize}

These consequences point to the fact that the low dimensional picture of getting
"stuck" in a high error local minima is mistaken, and that finding a local
minimum is actually a good thing. However, Newton's method (the core component
of all second order methods), is biased towards finding a critical point,
\textit{any} critical point. In the presence of an overwhelming number of
saddle points, it is likely that it will get stuck in one of them 
instead of going to a minimum. 

Let $f$ be our loss function, $\nabla f$ and $\matr{H}$ be it's gradient and 
Hessian respectively, and $\alpha$ our learning rate. 
The step taken by an algorithm at iteration $k$ is denoted by 
$\Delta \theta_k$. The property of Newton being driven towards a close critical
point can easily be seen by noting that its update equation
\begin{equation} \label{eq::Newton}
  \Delta \theta_k = - \alpha \matr{H}(\theta_k)^{-1} \nabla f(\theta_k)
\end{equation}
comes from taking a second order approximation of our loss function, and solving
for the closest critical point of this approximation 
(i.e. setting its gradient to 0).

To overcome this problem of Newton's method, \cite{Dauphin-ea-SP} proposes
a different algorithm, called saddle-free Newton, or SFN. 
The update equation for SFN is defined as
\begin{equation} \label{eq::SFN}
  \Delta \theta_k = - \alpha \abs{\matr{H}(\theta_k)}^{-1} \nabla f(\theta_k)
\end{equation}
The absolute value notation in equation \ref{eq::SFN} means that
$\abs{\matr{A}}$ is obtained by replacing the eigenvalues of $\matr{A}$ with 
their absolute values. In the convex case, this changes nothing from Newton.
However, in the
nonconvex case, this allows one to keep the very smart rescaling of Newton's
method, but still going on a descent direction when the eigenvalues are 
negative.

While saddle-free Newton showed great promise, its main problem is the
computational complexity it carries. 
Let $m$ be the number of parameters, or more
generally, the dimension of $f$'s domain.
The cost of calculating the update in equation \ref{eq::SFN} is the cost
of diagonalizing (and then inverting) the matrix $\abs{\matr{H}}$, 
namely $\mathcal{O}(m^3)$. Furthermore, this has a memory cost of 
$\mathcal{O}(m^2)$, because it needs to store the full Hessian. 
Since in neural network problems $m$ is typically bigger than $10^6$, 
both costs are prohibitive, which is the reason \cite{Dauphin-ea-SP} 
employs a low-rank approximation. Using a rank $k$
approximation, the algorithm has $\mathcal{O}(k^2m)$ time cost 
and $\mathcal{O}(km)$ memory cost. While this is clearly cheaper than the full
method, it's still intractable for current problems, since in order to get
a useful approximation the $k$ required becomes prohibitively large, especially
for the memory cost.

Another line of work is the one followed by Hessian-free optimization
\citep{Martens-HF}, popularly known as HF. This method centers in three core
ideas:
\begin{itemize}
  \item The Gauss-Newton method, which consists in replacing the use of the 
    Hessian in Newton's algorithm \eqref{eq::Newton}
    for the Gauss-Newton matrix $\matr{G}$.
    This matrix is a positive definite approximation of the Hessian, and it
    has achieved a good level of applicability in convex problems.
    
    However, the behaviour when the loss is
    nonconvex is not well understood. Furthermore, \cite{Mizutani-ea-NC}
    argues against using the Gauss-Newton matrix on neural networks, showing
    it suffers from poor conditioning and drops the negative curvature
    information, which is argued to be crucial. Note that this is a major
    difference with SFN, which leverages the negative curvature information,
    keeping the scaling in these directions.
  \item  Using conjugate gradients (CG) to solve the system 
    $\matr{G}(\theta_k)^{-1}\nabla f(\theta_k)$.
    One key advantage of CG is that it doesn't require to store
    $\matr{G}(\theta_k)$, only to calculate matrix-vector products of the form 
    $\matr{G}(\theta_k)v$ 
    for any vector $v$. The other advantage of this method is that
    it's iterative, allowing for early stopping when the solution to the system
    is good enough.
  \item  When using neural networks, 
    the $\mathcal{R}$-operator \citep{Pearlmutter-R, Schraudolph-FC} 
    is an algorithm to calculate matrix-vector products of the form
    $\matr{H}v$ and $\matr{G}v$ in $\mathcal{O}(m)$ time without storing
    any matrix. This is obviously very efficient,
    since normally multiplying an $m$-by-$m$ matrix with a vector has
    $\mathcal{O}(m^2)$ time and memory cost.
\end{itemize}

While Hessian-free optimization is computationally efficient, the use of 
the Gauss-Newton matrix in nonconvex objectives is thought to be inneffective.
The update equation of saddle-free Newton is specially designed for this kind
of problems, but current implementations lack computational efficiency.

In the following section, we propose a new algorithm that takes the advantages
of both approaches. This renders a novel second order method that's
computationally efficient, and specially designed for nonconvex optimization
problems.

\section{Saddle-free Hessian-free Optimization}

Something that comes to mind is the possibility of using 
conjugate gradients
to solve the system $\abs{\matr{H}}^{-1} \nabla f$ appearing in
equation \eqref{eq::SFN}. This would allow us to have an iterative method,
and possibly do early stopping when the solution to the system is good enough.
However, in order to do that we would need to calculate $\abs{\matr{H}}v$ for
any vector $v$. While this was easy with $\matr{H}v$ and $\matr{G}v$ via the
$\mathcal{R}$-operator, it doesn't extend to calculating
$\abs{\matr{H}}v$, so we arrive at an impass.

The first step towards our new method comes from the following simple but
important observation, that we state as a Lemma.
\begin{lemma}
 Let $\matr{H}$ be a real symmetric $m$-by-$m$ matrix. 
  Then, $\abs{\matr{H}}^2 = \matr{H}^2$.
\end{lemma}
\begin{proof}
  First we prove this for a real diagonal matrix $\matr{D}$. We denote
  $\matr{D}_{i,i} = \lambda_i$. By definition,
  we have that $\abs{\matr{D}}_{i,i} = \abs{\lambda_i}$ and it vanishes on the
  off-diagonal entries. Therefore, it is trivially verified that
  $(\abs{\matr{D}}^2)_{i,i} = \abs{\lambda_i}^2 = \lambda_i^2 = 
  \left(\matr{D}^2\right)_{i,i}$
  and both matrices are diagonal, which makes them coincide.

  Let $\matr{H}$ be a real symmetric $m$-by-$m$ matrix. 
  By the spectral
  theorem, there is a real diagonal matrix $\matr{D}$ and an orthogonal matrix
  $\matr{U}$ such that $\matr{H} = \matr{U} \matr{D} \matr{U}^{-1}$. Therefore,
  \begin{align*}
    \abs{\matr{H}}^2 &= \left(\matr{U} \abs{\matr{D}} \matr{U}^{-1} \right)^2
      = \matr{U} \abs{\matr{D}} \matr{U}^{-1}\matr{U} \abs{\matr{D}}  
        \matr{U}^{-1} \\
      &= \matr{U} \abs{\matr{D}}^2 \matr{U}^{-1}
      = \matr{U} \matr{D}^2 \matr{U}^{-1} \\
      &= \matr{U} \matr{D} \matr{U}^{-1}\matr{U} \matr{D} \matr{U}^{-1} 
      = \left( \matr{U} \matr{D} \matr{U}^{-1} \right)^2 \\
      &= \matr{H}^2
  \end{align*}
\end{proof}
Let $\matr{A}$ be a (semi-)positive definite matrix. Recalling that the square
root of $\matr{A}$ (noted as $\matr{A}^{\frac{1}{2}}$) is defined as the only
(semi-)positive definite matrix $\matr{B}$ such that $\matr{B}^2 = \matr{A}$,
we have the following corollary.
\begin{coro}
  Let $\matr{H}$ be a real square matrix. Then, $\abs{\matr{H}}$ is the square
  root of $\matr{H}^2$. Namely, $\abs{\matr{H}} = \left(\matr{H}^2\right)^
  \frac{1}{2}$.
\end{coro}
Note that our main impass is not knowing how to calculate $\abs{\matr{H}}v$
for any vector $v$. However, we know how to calculate $\matr{H}^2v = \matr{H}(
\matr{H}v)$ by applying the $\mathcal{R}$-operator twice. Therefore, the
problem can be reformulated as: given a positive definite matrix $\matr{A}$, of
which we know how to calculate $\matr{A}u$ for any vector $u$, can we calculate
$\matr{A}^{\frac{1}{2}}v$ for a given vector $v$?

The answer to this question is yes. As illustrated by \cite{Allen-ea-SQRT},
we can define the following initial value problem:
\begin{equation} \label{eq::difeq}
  \begin{cases}
    x'(t) = -\frac{1}{2}\left(t \matr{A} + (1 - t) \matr{I}\right)^{-1}
      \left(\matr{I} - \matr{A}\right)x(t) \\
    x(0) = v
  \end{cases}
\end{equation}
When the norm of $\matr{A}$ is small enough (which can be trivially rescaled), 
one can show that the ordinary
differential equation \eqref{eq::difeq} has the unique solution
\begin{equation*}
  x(t) = \left(t \matr{A} + (1-t) \matr{I}\right)^{\frac{1}{2}} v
\end{equation*}
This solution has the crucial property that $x(1) = \matr{A}^{\frac{1}{2}}v$.
Therefore, to calculate $\matr{A}^{\frac{1}{2}}v$ we can initialize $x(0) = v$,
and plug in equation \eqref{eq::difeq} to an ODE solver such as the different
Runge Kutta methods. The second core property of this formulation is that
in order to do the derivative evaluations required to solve the ODE, we only
need to multiply by $(\matr{I} - \matr{A})$ and solve systems by 
$\left(t \matr{A} + (1 - t) \matr{I}\right)$, both of which can be done
only with products of the form $\matr{A}u$ without storing any matrix, 
using conjugate gradients for the linear systems.

In order to solve our problem of approximating SFN in a Hessian-free way, we
could calculate $\abs{\matr{H}}v$ using $\matr{A}u := \matr{H}^2u$
in the previous
method and do conjugate gradients to solve the system in \eqref{eq::SFN}.
However, this would require solving
an ODE for every iteration of conjugate gradients, which would be quite
expensive. Therefore, we propose to calculate update \eqref{eq::SFN} in a
two-step manner. First, we multiply by $\abs{\matr{H}}$ and then we divide
by $\matr{H}^2 = \abs{\matr{H}}^2$:
\begin{align*}
  y &\leftarrow \abs{\matr{H}(\theta_k)} \nabla f(\theta_k) \\
  \Delta \theta_k &\leftarrow - \alpha (\matr{H}(\theta_k)^2)^{-1} y
\end{align*}
Combining this approach with our approximation schemes, we derive our
final algorithm, that we deem saddle-free Hessian-free optimization:
\begin{align*}
  y & \leftarrow \text{ODE-solve}\left(\text{Equation \eqref{eq::difeq}, 
    $\matr{A}u := \matr{H}(\theta_k)^2u$, 
    $v = \nabla f(\theta_k)$}
    \right) \\
  \Delta \theta_k & \leftarrow \text{CG-Solve}\left(
    \matr{H}(\theta_k)^2, -\alpha y \right)
\end{align*}
If $l$ is the number of Runge Kutta steps we take to solve the ODE 
\eqref{eq::difeq}, and $k$ is the number of CG iterations used to solve
the linear systems, 
then the overall cost of the algorithm is $\mathcal{O}(mlk)$.
Since $l$ is close to 20 in the successful experiments done 
by \cite{Allen-ea-SQRT} on
random matrices (independently of $m$), and $k$ is no larger than 250 in 
typical Hessian-free implementations, this
is substantially lower than the $\mathcal{O}(m^3)$ cost of saddle-free Newton.
Furthermore, one critical advantage is that the memory cost of the algorithm
is $\mathcal{O}(m)$, since at no moment it is required to store more than a
small constant number of vectors of size $m$.

\section{Conclusion and Future Work}
We presented a new algorithm called saddle-free Hessian-free optimization. 
This algorithm
provides a first step towards merging the benefits of computationally
efficient Hessian-free approaches and methods like saddle-free Newton,
which are specially designed for nonconvex objectives.

Further work will be focused on taking these ideas to real world applications,
and adding more speed and stability improvements to the core algorithm,
such as the preconditioners of \cite{Martens-HF, Martens-ea-CP} and damping
with Levenberg-Marquardt \citep{Levenberg1944, marquardt:1963}
style heuristics.

\section*{Acknowledgements}
The author would like to specially thank Marco Vanotti for his constant
and inconditional support.
The author would also like to thank Yoshua Bengio, Yann
Dauphin and Harm de Vries for very fruitful discussions.

\bibliography{sfhf_workshop}
\bibliographystyle{abbrv}

\end{document}